\documentclass[10pt]{article}

\usepackage{amsmath}    
\usepackage{graphicx}   
\usepackage{verbatim}   
\usepackage{color}      
\usepackage{subfigure}  

\usepackage{amssymb}
\usepackage{amsthm}

\theoremstyle{definition}
\newtheorem{Def}{Definition}[section]

\newtheorem{Prop}[Def]{Proposition}
\newtheorem{Lem}[Def]{Lemma}
\newtheorem{Thm}[Def]{Theorem}
\newtheorem{Cor}[Def]{Corollary}
\newtheorem{Rem}[Def]{Remark}
\newtheorem{Ass}[Def]{Assumption}

\newcommand{\p}{\mathbb{P}}
\newcommand{\e}{\mathbb{E}}
\newcommand{\real}{\mathbb{R}}
\newcommand{\n}{\mathbb{N}}

\usepackage{slashbox}

\allowdisplaybreaks

\begin{document}

\title{Strong Rate of Convergence for the Euler-Maruyama Approximation of Stochastic Differential Equations with Irregular Coefficients\footnote{This research has supported by grants of the Japaneses government.}}
\author
{
Hoang-Long Ngo\footnote{Hanoi National University of Education, 136 Xuan Thuy, Cau Giay, Hanoi, Vietnam,   email: $\qquad$ngolong@hnue.edu.vn},
 Dai Taguchi\footnote{Ritsumeikan University, 1-1-1 Nojihigashi, Kusatsu, Shiga, 525-8577, Japan, email: dai.taguchi.dai@gmail.com }
}
\date{}
\maketitle

\begin{abstract}
We consider the Euler-Maruyama approximation for multi-dimensional stochastic differential equations with irregular coefficients. 
We provide the rate of strong convergence where the possibly discontinuous drift coefficient satisfies a one-sided Lipschitz condition and the diffusion coefficient is H\"older continuous and uniformly elliptic.
\end{abstract}

\section{\large Introduction}
Let us consider the $d$-dimensional stochastic differential equation (SDE)
\begin{align}\label{SDE_1}
X_t=x_0+ \int_0^t b(s, X_s)ds + \int_0^t \sigma(s, X_s)dW_s, \:x_0 \in \real^d,\: t \in [0,T],
\end{align}
where $W:=(W_t)_{0\leq t \leq T}$ is a standard $d$-dimensional Brownian motion on a probability space $(\Omega, \mathcal{F},\p)$ with a filtration $(\mathcal{F}_t)_{0\leq t \leq T}$ satisfying the usual conditions.
The drift coefficient $b$ is a Borel-measurable function from $[0,T] \times \real^d$ into $\real^d$ and the diffusion coefficient $\sigma$ is a Borel-measurable function from $[0,T] \times \real^d$ into $\real^d \times \real^d$.
In this article, we consider that elements of $\real^d$ are column vectors.
The diffusion process $X:=(X_t)_{0 \leq t \leq T}$ is used to model many random dynamical phenomena in many fields of application.

Since the solution of (\ref{SDE_1}) is rarely analytically tractable, one often approximates $X$ by using the Euler-Maruyama scheme given by 
\begin{align}\label{EM_1}
   X_t^{(n)} 
&= x_0 +\int_0^tb\left(\eta_n(s), X_{\eta _n(s)}^{(n)}\right)ds +\int_0^t \sigma\left(\eta _n(s),X_{\eta _n(s)}^{(n)}\right) dW_s,\:\: t \in [0,T],
\end{align}
where $\eta _n(s) = kT/n=:t_k^{(n)}$ if $ s \in \left[kT/n, (k+1)T/n \right)$.
It is well-known that if the coefficients $b$ and $\sigma$ are Lipschitz continuous in space and $1/2$-H\"older continuous in time then the Euler-Maruyama scheme is known to have a strong rate of convergence order $1/2$, i.e. for any $p>0$, there exists $C_p>0$ such that
$$\e \Big[ \sup_{0 \leq t \leq T}|X_t - X^{(n)}_t|^p\Big] \leq C_p n^{-p/2}.$$

The strong rate in the case of non-Lipschitz coefficients have been studied recently by
using the approximation method of Yamada and Watanabe (\cite{Yamada}, Theorem 1) in Gy\"ongy and R\'asonyi \cite{Gyongy}. 
They have proven that  for a one-dimensional SDE, if the diffusion coefficient is $(\alpha + 1/2)$-H\"older continuous in space and the drift is the sum of a Lipschitz and a monotone decreasing $\gamma$-H\"older continuous function then 
\begin{align}\label{result_GR}
\e \Big[  \sup_{0 \leq t \leq T}|X_t - \tilde{X}^{(n)}_t| \Big] \leq \begin{cases} C (\log n)^{-1/2} & \text{ if } \alpha = 0\\ C( n^{-2\alpha^2} + n^{-\alpha \gamma}) &\text{ if } \alpha \in (0,1/2], \end{cases}
\end{align}
where $\tilde{X}$ is the Euler's ``polygonal" approximation of $X$ given by
\begin{align}\label{EM_2}
   \tilde{X}^{(n)}_t
&= x_0 +\int_0^tb\left(s, \tilde{X}_{\eta _n(s)}^{(n)}\right)ds +\int_0^t \sigma\left(s,\tilde{X}_{\eta _n(s)}^{(n)}\right) dW_s,\:\: t \in [0,T].
\end{align}
 Yan \cite{Y} has obtained a similar result to (\ref{result_GR}) for the Euler-Maruyama scheme for the one-dimensional SDE with drift which is Lipschitz continuous in space and H\"older continuous in time by using Tanaka's formula and some estimates for the local time. 
When the drift $b$ is not supposed to be continuous, Halidias et al. (Theorem 3.1 in \cite{HK}) have shown the convergence of Euler-Maruyama approximation in $L^2$-norm (see also Theorem 2.8 in \cite{GK}). Regarding the rates of convergence,  Gy\"ongy  has shown that if $b$ satisfies the one-sided Lipschitz condition (see Definition \ref{OSL_1}) and $\sigma$ is locally Lipschitz then the rate of almost convergence for the Euler-Maruyama's polygonal approximation is of order $1/4$ (see \cite{G98} Theorem 2.6).  Moreover, Bastani et al. have recently proven strong $L^p$-rate $1/4$ for $p\geq 2$ for split-step backward Euler approximations of SDEs with discontinuous drift and Lipschitz continuous diffusion coefficients (see Theorem 5.2 in \cite{BT}). 

Besides the strong approximation problem, the weak approximation  for non-Lipschitz coefficients SDE has also received a lot of attention. The weak rate of the Euler-Maruyama approximation when both drift and diffusion coefficients as well as payoff functions  are H\"older continuous has been studied in \cite{KP, MP, Gyongy}. Kohatsu-Higa et al. studied weak approximation errors for SDE with discontinuous drift by using a perturbation method in \cite{KLY}. The case of locally Lipschitz coefficients has  been studied extensive too, see \cite{HJK} and the references therein.  It should be noted that 
the strong rate of approximation is very useful to implement an effective Multi-level Monte Carlo simulation scheme for approximating expectation of some functionals of $X$ (see \cite{G}).


The goal of this article is to show  that the strong rates obtained in \cite{Gyongy} and \cite{Y} still hold even when $b$ is discontinuous. More precisely, we will investigate the strong rate of the Euler-Maruyama approximation under  the assumption that the diffusion coefficient $\sigma$ is $(\alpha + 1/2)$-H\"older continuous and the drift $b$ is the one-sided Lipschitz and belongs to the class $\mathcal{A}$ of functions which is, roughly speaking, of bounded variation with respect to a Gaussian measure on $\real^d$.  
In particular, our result implies that the Euler-Maruyama approximation has the optimal strong rate $1/2$ in the case of Lipschitz continuous diffusion coefficient and discontinuous drift. Hence our result partly improves the ones in \cite{G98, Gyongy, Y}. 
In this article, Lemma \ref{Lem} is the key estimation. If the drift coefficient $b$ is a Lipschitz continuous function, it is easy to prove this lemma. To obtain the same estimate with discontinuous drift, we use the result of Lemaire and Menozzi which is the Gaussian bound for the density of the Euler-Maruyama approximation (see \cite{Lemaire}, Theorem 2.1).

Finally we note that SDEs with discontinuous drift appear in many applications such as optimal control and interacting infinite particle systems, see e.g. \cite{BSW, CS, KR}. 

Our paper is divided as follows: Section \ref{sec:results} introduces some notations and assumptions for our framework together with the main results. All proofs are deferred to Section \ref{sec:proofs}.


\section{\large Main results}\label{sec:results}


\subsection{Notations and Assumption}
We first introduce the class of functions for the drift coefficient.

\begin{Def}
Let $\mathcal{A}$ be a class of all  bounded measurable functions $\zeta: [0,T] \times \real^d \rightarrow \real$ such that there exists a sequence of functions $(\zeta_N(t,\cdot))_{N\in\n} \subset C^1(\real^d;\real)$ satisfying the following conditions:
\begin{itemize}
\item[$\mathcal{A}$(i)]   For any $L >0$, $\displaystyle \sup_{t \in [0,T]} \int_{|x|\leq L} |\zeta_N (t,x) - \zeta(t,x)|dx \rightarrow 0 $ as $N \to \infty$.
\item[$\mathcal{A}$(ii)]  There exists a positive constant $K$ such that for any $x \in \real^d$,
$$\displaystyle \sup_{t \in [0,T]} \sup_{N\in \n}|\zeta_N(t,x)| \leq K.$$
\item[$\mathcal{A}$(iii)] There exists a positive constant $K$ such that for any $a \in \real^d$ and $u>0$, 
$$\displaystyle \sup_{t \in [0,T]} \sup_{N \in \n} \sum_{i=1}^d\int_{\real^d}\left|\partial_i \zeta_N(t,x+a) \right| \frac{e^{-\frac{|x|^2}{u}}}{u^{(d-1)/2}}dx \leq K(1+ \sqrt{u}),$$
where $\partial_i$ is partial derivative in space.
\end{itemize}
\end{Def} 
We call $(\zeta_N)_{N \in \n}$ an $\mathcal{A}$-approximation sequence of $\zeta$. This class of function $\mathcal{A}$ is similar to the one introduced in \cite{KMN}. The following proposition shows that this class is quite large. Its proof is deferred to  Section \ref{Appe}. 

\begin{Prop}\label{exp_1}\rm 
(i) If $\xi, \zeta \in \mathcal{A}$ and $\alpha, \beta \in \real$, then $\xi \zeta \in \mathcal{A}$ and $\alpha \xi + \beta \zeta \in \mathcal{A}$.\\
(ii) If $g:[0,T]\times \real^d \rightarrow \real$ is  a bounded measurable function and $g(t, \cdot):\real^d \rightarrow \real$ is monotone in each variable separately then $g \in \mathcal{A}$.\\
(iii) If  $g$ is bounded and Lipschitz continuous in space, then $g \in \mathcal{A}$.
\end{Prop} 
Using Proposition \ref{exp_1} one can easily verify that the class $\mathcal{A}$ contains function $\zeta(x) = |x-a|\wedge 1$ or $\zeta(x) = I_{a < x < b}$ for some $a, b \in \real^d$. 

\begin{Def}\label{OSL_1} 
A function $f :[0,T]\times \real^d \rightarrow \real^d$ is called \emph{one-sided Lipschitz function} in space if there exists a positive constant $K$ such that for any $(t,x, y) \in [0,T]\times \real^{2d}$,
\begin{align}\label{OLS_2}
\langle x-y, f(t,x)- f(t,y) \rangle_{\real^d} \leq K|x-y|^2.
\end{align}
Let $\mathcal{L}$ be the class of all one-sided Lipschitz functions.  
\end{Def} 

\begin{Rem}\label{OSL_3}\rm 
By the definition of the class $\mathcal{L}$, if $f, g \in \mathcal{L}$ and $\alpha \geq 0$, then $f + g$, $ \alpha f \in \mathcal{L}$. 
The one-sided Lipschitz property is closely related to the monotonicity condition introduced in \cite{G98} and the class $\mathcal{L}$ obviously contains all functions which are the sum of a Lipschitz function and a monotone decreasing $\gamma$-H\"older continuous function considered in \cite{Gyongy}. 
\end{Rem} 

Many properties and applications of SDEs with the one-sided Lipschitz drift can be found in \cite{Sch}.
We need the following assumptions on the coefficients $b=(b^{(1)}, \cdots, b^{(d)})^{*}$ and $\sigma=(\sigma_{i,j})_{1\leq i,j\leq d}$. Here $^*$ means transpose for the matrix.
\begin{Ass} \label{Ass_1}\rm 
We assume that the coefficients $b$ and $\sigma $ are measurable functions and satisfy the following conditions:
\begin{itemize}
\item[(i)]   $b \in \mathcal{L}$ and $b^{(i)} \in \mathcal{A}$ for any $i = 1, \ldots, d$.
\item[(ii)]  $a= \sigma \sigma^{*}$ is uniformly elliptic, i.e., there exists $\lambda_0 \geq 1$ such that for any $(t,x,\xi) \in [0,T]\times \real^{2d}$, 
\begin{align*}
\lambda_0^{-1} |\xi|^2 \leq \langle a(t,x)\xi,\xi\rangle _{\real^d} \leq \lambda_0|\xi|^2.
\end{align*}
\item[(iii)]  $\sigma$ is a $(1/2 + \alpha) $-H\"older continuous with $\alpha \in [0, 1/2]$ in space, i.e., there exists $K > 0$ such that 
\begin{align*} 
      \sup_{(t,x,y)\in [0,T]\times \real^{2d}, x \neq y}\frac{|\sigma(t,x)-\sigma(t,y)| }{|x-y|^{1/2+\alpha}} \leq K.
\end{align*}

\end{itemize}
\end{Ass} 

\begin{Rem}
Many functions satisfy Assumption \ref{Ass_1} (i). For example, monotone decreasing function or Lipschitz continuous function.
In particular, for $x \in \real$, the function ${\bf 1}_{(-\infty, 0]}(x) - {\bf 1}_{(0, +\infty)}(x)$ satisfies Assumption (\ref{Ass_1}) (i). This function is the optimal drift coefficient for some stochastic control problem (see \cite{BSW} or \cite{KS} page 437).
From Remark \ref{exp_1} and \ref{OSL_3}, we know that if $ f $ and $ g $ satisfy Assumption \ref{Ass_1} (i) and $\alpha, \beta \geq 0$, then $\alpha f + \beta g$ also satisfies this condition.
Assumption \ref{Ass_1} (ii) implies that the diffusion coefficient $\sigma$ is bounded i.e., for any $(t,x) \in [0,T]\times \real^d$, $|\sigma(t,x)| = \{ \sum_{i,j} \sigma_{i,j}^2(t,x)\}^{1/2} \leq \sqrt{d\lambda_0}$.
\end{Rem}

\begin{Ass} \label{Ass_1b} The coefficients $b$ and $\sigma$ are $\beta$-H\"older continuous with $\beta \geq 1/2$ in time i.e., there exist $K>0$ such that for all $t,s \in [0,T]$ and $x \in \real^d$,
$$|b(t,x) - b(s,x)| + |\sigma(t,x) - \sigma(s,x)| \leq K|t-s|^{\beta}.$$
\end{Ass}
\begin{Rem} 
Veretennikov \cite{V} has shown the following result. Assume that $b$ and $\sigma$ are bounded measurable functions such that $\sigma \sigma^{*}$ is uniformly elliptic. If $\sigma$ is $1/2$-H\"older continuous in $x \in \real$  when $d = 1$ and it is Lipschitz in $x \in \real^d$ when $d \geq 2$, then there exists a unique strong solution to the stochastic differential equation (\ref{SDE_1}) (see also \cite{CE, GK, KR, Krylov, Zv} for other criteria for the existence and uniqueness of solution of SDE with non-Lipschitz coefficients).
\end{Rem} 

\subsection{Main Results}

Through the whole of this article, we will use the positive constants $C$ and $c$ which do not depend on $n$. 
Unless explicitly stated otherwise, the constant $C$ depends only on $K,T,\lambda_0,x_0,\beta$ and $d$, the constant $c$ depends only on $K, \lambda_0, \eta$ and $d$.
Moreover the constants $C$ and $c$ may change from line to line.

We obtain the following results on the rates of the Euler-Maruyama approximation in both $L^1$-norm and $L^1$-$\sup$ norm.


\begin{Thm} \label{Main_1}\rm 
Let Assumptions \ref{Ass_1} and \ref{Ass_1b} hold. Then there exists a constant $C$ such that, for $d=1$, 
\begin{align} \label{est_l1}
\sup_{\tau \in \mathcal{T}}\e[|X_{\tau}-X_{\tau}^{(n)}|] 
&\leq \left\{ \begin{array}{ll}
\displaystyle C(\log n)^{-1} &\textit{if } \alpha = 0,  \\
\displaystyle Cn^{-\alpha} &\textit{if } \alpha \in (0, 1/2],
\end{array}\right.
\end{align}
and for $d\geq 2$,
\begin{align*}
\sup_{\tau \in \mathcal{T}}\e[|X_{\tau}-X_{\tau}^{(n)}|] \leq Cn^{-1/2} &\textit{ if } \alpha = 1/2.
\end{align*}
where $\mathcal{T}$ is the set of all stopping times $\tau \leq T$. 
\end{Thm}

\begin{Thm} \label{Main_2} \rm 
Under Assumptions \ref{Ass_1} and \ref{Ass_1b} we have for $d=1$,
\begin{align} \label{est_l1sup}
\e\left[\sup _{0\leq t \leq T}|X_t - X_t^{(n)}|\right]
&\leq \left\{ \begin{array}{ll}
\displaystyle C(\log n)^{-1/2} &\textit{if } \alpha = 0,  \\
\displaystyle Cn^{-2\alpha^2} &\textit{if } \alpha \in (0, 1/2],
\end{array}\right.
\end{align}
and for $d\geq 2$,
\begin{align*} 
\e\left[\sup _{0\leq t \leq T}|X_t - X_t^{(n)}|\right]
\leq Cn^{-1/2} &\textit{  \;\;\;\;if } \alpha = 1/2.
\end{align*}
\end{Thm}
The following theorem provides the bound of the error in $L^p$-norm which is useful to design a Multi-level Monte Carlo approximation scheme. 
\begin{Thm} \label{Main_3} 
Let Assumptions \ref{Ass_1} and \ref{Ass_1b} hold. Then for any $p \geq 2$, there exists a constant $C=C(K,T,\lambda_0,x_0,d,\beta,p)$ such that for $d=1$,
\begin{align} \label{est_lp}
\e\left[\sup_{0\leq t \leq T}\left|X_t-X_t^{(n)}\right|^p\right] 
&\leq \left\{ \begin{array}{ll}
\displaystyle C\left( \log n \right)^{-1} &\textit{if $ \alpha = 0$},  \\
\displaystyle Cn^{-\alpha} &\textit{if }\alpha \in (0, 1/2],
\end{array}\right.
\end{align}
and for $d\geq 2$,
\begin{align*} 
     \e\left[\sup_{0\leq t \leq T}\left|X_t-X_t^{(n)}\right|^p\right] 
\leq Cn^{-1/2} &\textit{  \;\;\;\;if } \alpha = 1/2.
\end{align*}
\end{Thm}

\begin{Rem}
In Theorem \ref{Main_3}, for $\alpha \in [0,1/2)$, the result is the same as in Gy\"ongy and R\'asonyi \cite{Gyongy}. But for $\alpha =1/2$, every moment bigger than $2$ of the error is of the same order. The reason is that we deal with the discontinuous drift coefficients and the estimate of discontinuous part is of order $1/2$ for any $q \geq 1$ (see Lemma \ref{Lem}).
The proof of Theorem \ref{Main_3} does not use the result of Theorem \ref{Main_2} and only use the result of Theorem \ref{Main_1}. On the other hand, for $\alpha \in (0,1/2)$, using Theorem \ref{Main_3} and Jensen's inequality, we can obtain that the rate of convergence order is $\alpha/2$ in $L^1$-sup norm. For $\alpha \in [1/4, 1/2)$, this result is better than Theorem \ref{Main_2} and for $\alpha \in (0, 1/4]$, this result is worse than Theorem \ref{Main_2}.  
\end{Rem}

In many applications such as the regime switching problem, Assumption \ref{Ass_1b} fails to be satisfied. However we are still able to obtain the same strong rates of convergence as above if we consider the polygonal Euler-Maruyama scheme (\ref{EM_2}) instead of the original Euler-Maruyama scheme (\ref{EM_1}).

\begin{Cor} \label{Main_4}
Assume $d=1$ and Assumption \ref{Ass_1}. Then all the estimates (\ref{est_l1}), (\ref{est_l1sup}) and (\ref{est_lp}) still hold when we replace $X_t^{(n)}$ with $\tilde{X}^{(n)}_t$.
\end{Cor} 

\begin{Cor} \label{Main_5}
Let Assumption \ref{Ass_1} and \ref{Ass_1b} be satisfied. Then the conclusion of Theorem \ref{Main_1}, \ref{Main_2} and  \ref{Main_3} still hold if we replace $X_t^{(n)}$ with  $\overline{X}^{(n)}_t$  defined by
\begin{align*}
   \overline{X}^{(n)}_t
&= x_0 +\int_0^tb\left(s, \overline{X}_{\eta _n(s)}^{(n)}\right)ds +\int_0^t \sigma\left(\eta_n(s),\overline{X}_{\eta _n(s)}^{(n)}\right) dW_s,\:\: t \in [0,T].
\end{align*}
\end{Cor} 

\section{Proof of the Main Theorems} \label{sec:proofs}
\subsection{Gaussian bound for the density of the Euler-Maruyama scheme}
Under Assumption \ref{Ass_1} (ii) and (iii), it follows from (\cite{Lemaire}, Theorem 2.1) that the transition density $p^{(n)}(s,t,x,x')$ of $X_{t}^{(n)}$ between times $s$ and $t$ exists and there exist constants $C\geq 1$ and $c>0 $ such that for any $x,x' \in \real^d$ and $0 \leq j < j' \leq n$, 
\begin{align}\label{GB_1} 
     C^{-1} p_{c^{-1}}\left(t_{j'}^{(n)}-t_{j}^{(n)},x,x'\right) 
\leq p^{(n)}\left(t_{j}^{(n)},t_{j'}^{(n)},x,x'\right) 
\leq C p_{c}\left(t_{j'}^{(n)}-t_{j}^{(n)},x,x'\right),
\end{align}
where $\displaystyle p_c(t-s,x,x'):= \left(\frac{c}{2\pi (t-s)}\right)^{d/2} \exp\left( -c\frac{|x'-x|^2}{2(t-s)}\right)$.
Note that the constant $C$ depends on $K, \lambda_0, \eta,d,T$ and the constant $c$ depends on $K, \lambda_0, \eta,d$ but not on $T$.

The following lemma plays a crucial role in our argument.
\begin{Lem}\label{GB_3} 
Assume that Assumption \ref{Ass_1} (ii) and (iii) hold. Then there exist $C\geq 1$ and $c>0 $ such that for any $x,x' \in \real^d$, $0 \leq j < j' \leq n$ and $t \in (t_{j'-1}^{(n)}, t_{j'}^{(n)}]$, we have
\begin{align*}
      p^{(n)}\left(t_{j}^{(n)},t,x,x'\right) 
&\leq C p_{c}\left(t-t_{j}^{(n)},x,x'\right).
\end{align*}
\end{Lem} 
\begin{proof} \rm 
Note that for any $u \in \real^d$,
\begin{align*}
  &p^{(n)}\left(t_{j'-1}^{(n)},t,u,x'\right)
= \left(\frac{1}{2\pi}\right)^{d/2} \frac{1}{\sqrt{(t-t_{j'-1}^{(n)})^d\det a(t_{j'-1}^{(n)},u)}}\\
&\times \exp\left[-\frac{\langle a^{-1}(t_{j'-1}^{(n)},u)(x'-u-(t-t_{j'-1}^{(n)})b(t_{j'-1}^{(n)},u)),(x'-u-(t-t_{j'-1}^{(n)})b(t_{j'-1}^{(n)},u)) \rangle_{\real^d}}{2(t-t_{j'-1}^{(n)})}\right].
\end{align*}
Since $a^{-1}$ is uniformly elliptic, using the inequality $|x-y|^2\geq \frac{1}{2}|x|^2 - |y|^2$ for any $x,y \in \real^d$, we obtain
\begin{align*} 
& -\frac{\langle a^{-1}(t_{j'-1}^{(n)},u)(x'-u-(t-t_{j'-1}^{(n)})b(t_{j'-1}^{(n)},u)),(x'-u-(t-t_{j'-1}^{(n)})b(t_{j'-1}^{(n)},u)) \rangle_{\real^d}}{2(t-t_{j'-1}^{(n)})} \\
&\leq -\frac{\lambda_0^{-1} |x'-u-(t-t_{j'-1}^{(n)})b(t_{j'-1}^{(n)},u)|^2}{2(t-t_{j'-1}^{(n)})} \\
&\leq -\frac{(2\lambda_0)^{-1} |x'-u	|^2}{2(t-t_{j'-1}^{(n)})} + \frac{\lambda_0^{-1} (t-t_{j'-1}^{(n)})^2|b(t_{j'-1}^{(n)},u)|^2}{2(t-t_{j'-1}^{(n)})}  
 \leq -c\frac{|x'-u	|^2}{2(t-t_{j'-1}^{(n)})}+ C.\\
\end{align*}
Hence we have
\begin{align*}
      p^{(n)}\left(t_{j'-1}^{(n)},t,u,x'\right)
\leq  C \frac{\exp\left[-c\frac{|x'-u|^2}{2(t-t_{j'-1}^{(n)})} \right]}{\left(2\pi(t-t_{j'-1}^{(n)})\right)^{d/2} }.
\end{align*}
This estimate together with the Chapman-Kolmogorov equation and (\ref{GB_1}) yield	
\begin{align*} 
      &p^{(n)}\left(t_{j}^{(n)},t,x,x'\right) 
 =\int_{\real^d} p^{(n)}\left(t_{j}^{(n)},t_{j'-1}^{(n)},x,u\right) p^{(n)}\left(t_{j'-1}^{(n)},t,u,x'\right) du\\
&\leq C \int_{\real^d}
 \frac{\exp\left[-c\frac{|u-x|^2}{2(t_{j'-1}^{(n)} - t_{j}^{(n)})} \right]}{\left(2\pi(t_{j'-1}^{(n)} - t_{j}^{(n)})\right)^{d/2} } 
\frac{\exp\left[-c\frac{|x'-u|^2}{2(t-t_{j'-1}^{(n)})} \right]}{\left(2\pi(t-t_{j'-1}^{(n)})\right)^{d/2} }du 
= C p_{c}\left(t-t_{j}^{(n)},x,x'\right).
\end{align*}
We therefore obtain the desired estimate.
\end{proof}

\begin{Cor}
Let $p_t^{(n)}$ be a density for $X_t^{(n)}$. From Lemma \ref{GB_3} with $j=0$, there exist $C \geq 1$ and $c>0$ such that for any $t \in (0,T] $ and $x \in \real^d$ we have 
\begin{align}\label{GB_EM} 
p_t^{(n)}(x) = p^{(n)}( 0,t ,x_0, x ) \leq C p_c(t,x_0,x).
\end{align}
\end{Cor}

\subsection{Some auxiliary estimates}
In this section, we give a key estimation (Lemma \ref{Lem}) to prove the main theorems. 

\begin{Lem}\label{GB_2} 
Let $ \zeta: [0,T] \times \mathbb{R}^d \to \mathbb{R}$ be a bounded measurable function and $(\zeta_N)_{N \in\n}$ be a sequence of functions satisfying $\mathcal{A}$(i) for $\zeta$.
Let $(Y_t)_{ 0  \leq t \leq T }$ be 
a $d$-dimensional stochastic process with $Y_0=y_0 \in \real^d$.
Suppose that $Y_t$ satisfies the Gaussian bound condition on $[\kappa,T]$ for some $\kappa \in (0,T]$, i.e., there exist positive constants $C$ and $c$ such that  
\begin{equation} \label{GB_Y} 
p_t(y) \leq C \frac{e^{\frac{-c|y-y_0|^2}{t}}}{t^{d/2}}, \quad t\in [\kappa,T],
\end{equation}
where $p_t$ is the density function of $Y_t$. Then 
\begin{align}\label{GB_2_1} 
\int_{\kappa}^T \e\left[\left|\zeta_N(t,Y_t) - \zeta(t,Y_t)\right|\right]dt \rightarrow 0,\: N \rightarrow \infty.
\end{align}
and if $T/n \geq \kappa$
\begin{align}\label{GB_2_2} 
\int_{\frac{T}{n}}^T \e\left[\left|\zeta_N(t,Y_{\eta_n(t)}) - \zeta(t,Y_{\eta_n(t)})\right|\right]dt \rightarrow 0, \: N \rightarrow \infty.
\end{align}
\end{Lem} 
\begin{proof} \rm 
For given $\varepsilon > 0 $, there exists $M \equiv M(\varepsilon,y_0,c) >0$ such that for any $|y| \geq M$,
\begin{align}  \label{boundeps}
e^{-\frac{c|y-y_0|^2}{2t}} \leq e^{-\frac{c|y-y_0|^2}{2T}} < \varepsilon.
\end{align}
From $\mathcal{A}$(i), there exists $N' \equiv N'	(\varepsilon)$ such that for any $N \geq N'$,
\begin{align*}
\int_{|y|<M} |\zeta_N (t,y) - \zeta(t,y)|dy < \varepsilon.
\end{align*}
Therefore for any $N \geq N'$, using the Gaussian bound condition  (\ref{GB_Y}), the uniform boundedness of $\zeta_N$ and $\zeta$, and (\ref{boundeps}), we get
\begin{align*}
      &\int_{\kappa}^T|\e[\zeta_N(t,Y_t) - \zeta(t,Y_t)]| dt
\leq  \int_{\kappa}^T dt \int_{\real^d}dy |\zeta_N(t,y) - \zeta(t,y)| p_t(y)\\
&\leq C \int_{\kappa}^T dt \left( \int_{|y|<M} dy + \int_{|y| \geq M} dy \right)|\zeta_N(t,y) - \zeta(t,y)| \frac{e^{\frac{-c|y-y_0|^2}{t}}}{t^{d/2}}\\
&\leq  C\int_{\kappa}^T \frac{1}{t^{d/2}} dt \int_{|y|<M} dy |\zeta_N(t,y) - \zeta(t,y)|
 +C \int_{\kappa}^T dt \int_{|y|\geq M} dy \frac{e^{\frac{-c|y-y_0|^2}{t}}}{t^{d/2}}\\
&\leq CT\frac{1}{\kappa^{d/2}} \varepsilon + CT\varepsilon \int_{\real^d}  \frac{e^{\frac{-c|y-y_0	|^2}{2t}}}{t^{d/2}}dy \leq C_{T,\kappa} \varepsilon,
\end{align*}
where the constant $C_{T,\kappa}$ depending only on $T$ and $\kappa$.
Hence by letting $\varepsilon $ to $0$, we conclude the proof of (\ref{GB_2_1}).
In the same way, we can show (\ref{GB_2_2}).
\end{proof}

\begin{Cor}\label{exp_3}
Let Assumption \ref{Ass_1} hold and $(b_N^{(i)})$ be an $\mathcal{A}$-approximation sequence of $b^{(i)}$ for each $i=1, \ldots, d$.  Then for any $i=1,\cdots,d$ and $n \in \n$, we have 
\begin{align*} 
   \lim_{N\rightarrow \infty}\int_{\frac{T}{n}}^T\e[| b^{(i)}_N(s, X_s^{(n)}) - b^{(i)}_N(s, X_{\eta_n(s)}^{(n)})|]ds
 = \int_{\frac{T}{n}}^T\e[| b^{(i)}(s, X_s^{(n)}) - b^{(i)}(s, X_{\eta_n(s)}^{(n)})|]ds < +\infty.
\end{align*}
\end{Cor}
\begin{proof} \rm 
It follows from Lemma \ref{GB_3} that the densities of $X_s^{(n)}$ and $X_{\eta_n(s)}^{(n)}$ satisfy the Gaussian bound condition for $s\geq \frac Tn$. Hence using Lemma \ref{GB_2} with $\kappa = T/n$ and the simple inequality, $| |a-b| - |a'-b'| | \leq |a-a'|+|b-b'|$, we have
\begin{align*} 
&     \left| \int_{\frac{T}{n}}^T\e[| b^{(i)}_N(s, X_s^{(n)}) - b^{(i)}_N(s, X_{\eta_n(s)}^{(n)})|]ds
 - \int_{\frac{T}{n}}^T\e[| b^{(i)}(s, X_s^{(n)}) - b^{(i)}(s, X_{\eta_n(s)}^{(n)})|]ds \right| \\
&\leq \int_{\frac{T}{n}}^T\e \left[\left| | b^{(i)}_N(s, X_s^{(n)}) - b^{(i)}_N(s, X_{\eta_n(s)}^{(n)})| - | b^{(i)}(s, X_s^{(n)}) - b^{(i)}(s, X_{\eta_n(s)}^{(n)})| \right|\right]ds  \\
&\leq \int_{\frac{T}{n}}^T\e \left[| b^{(i)}_N(s, X_s^{(n)}) - b^{(i)}(s, X_s^{(n)}) |\right]ds + \int_{\frac{T}{n}}^T\e \left[| b^{(i)}_N(s, X_{\eta_n(s))}^{(n)}) - b^{(i)}(s, X_{\eta_n(s))}^{(n)} |\right]ds\\
& \rightarrow 0, \text{ as } N \to \infty,
\end{align*}
which implies the desired result.
\end{proof}

The above corollary is useful to prove the following key estimate.

\begin{Lem} \label{Lem} \rm 
Let $b^{(i)} \in \mathcal{A}$ for $i=1, \cdots d$. Under Assumption \ref{Ass_1} (ii) and (iii), for any $q \geq 1$, there exists $C \equiv C(K,T,\lambda_0,x_0,d,q)$ such that
\begin{align} \label{bxsbxeta}
\sum_{i=1}^d \int_0^T\e[| b^{(i)}(s,X_s^{(n)}) - b^{(i)}(s,X_{\eta_n(s)}^{(n)})|^q]ds \leq \frac{C}{\sqrt{n}}.
\end{align}
\end{Lem}

\begin{Rem}
The bound (\ref{bxsbxeta}) is tight. Indeed, let us consider the case $d=1$, $x_0=0$, $\sigma=1$ and $b(x)={\bf 1}_{(-\infty,0]}(x)-{\bf 1}_{(0,+\infty)}(x)$. 
We will show that for any $q>0$, it holds
\begin{align} \label{add}
\int_0^T\e[| b(X_s^{(n)}) - b(X_{\eta_n(s)}^{(n)})|^q]ds \geq \frac{C}{\sqrt{n}}
\end{align}
for some constant $C>0$. Indeed, for any $s \geq T/n$, since $X_{\eta_n(s)}^{(n)}$ and $W_s-W_{\eta_n(s)}$ are independent,
\begin{align*}
  &\e[| b(X_s^{(n)}) - b(X_{\eta_n(s)}^{(n)})|^q] =2^{q-1} \e[| b(X_s^{(n)}) - b(X_{\eta_n(s)}^{(n)})|]\\ 
&=  2^q\p \left((X_{\eta_n(s)}^{(n)}+(s-\eta_n(x)) b(X_{\eta_n(s)}^{(n)})+ W_s-W_{\eta_n(s)}) X_{\eta_n(s)}^{(n)}<0 \right)\\
&=2^q \int_{\real}dx \int_{\real}  {\bf 1}(x^2+(s-\eta_n(s))b(x)x + yx <0) p_{\eta_n(s)}^{(n)}(x) \frac{e^{-\frac{y^2}{2(s-\eta_n(s))}}}{\sqrt{2 \pi (s-\eta_n(s))}} dy\\
& \geq \int_0^{\infty}dx \int_{\real}  {\bf 1}(x-(s-\eta_n(s)) + y <0) p_{\eta_n(s)}^{(n)}(x) \frac{e^{-\frac{y^2}{2(s-\eta_n(s))}}}{\sqrt{2 \pi (s-\eta_n(s))}}dy.
\end{align*}
Let $\Phi(u):=\int_{-\infty}^u\frac{e^{-\frac{v^2}{2}}}{\sqrt{2\pi}}dv$.
Then by the change of variable $z = y/\sqrt{s-\eta_n(s)}$ , we have
\begin{align*}
  &\e[| b(X_s^{(n)}) - b(X_{\eta_n(s)}^{(n)})|^q]  \geq  \int_0^{\infty} p_{\eta_n(s)}^{(n)}(x) \Phi \left(-\frac{x-(s-\eta_n(s))}{\sqrt{s-\eta_n(s)}}\right)dx.
\end{align*}
Recall that $x_0=0$.
It follows from the lower bound of \eqref{GB_1} that 
\begin{align*}
  &\e[| b(X_s^{(n)}) - b(X_{\eta_n(s)}^{(n)})|^q] 
\geq \frac{1}{C} \int_0^{\infty} \frac{e^{-\frac{x^2}{2c\eta_n(s)}}}{\sqrt{2\pi c \eta_n(s)}} \Phi \left(-\frac{x-(s-\eta_n(s))}{\sqrt{s-\eta_n(s)}}\right)dx\\
&\geq \frac{1}{C} \int_0^{\sqrt{s-\eta_n(s)}} \frac{e^{-\frac{x^2}{2c\eta_n(s)}}}{\sqrt{2\pi c \eta_n(s)}} \Phi \left(-\frac{x-(s-\eta_n(s))}{\sqrt{s-\eta_n(s)}}\right) dx\\
&\geq \frac{1}{C} \int_0^{\sqrt{s-\eta_n(s)}} \frac{e^{-\frac{s-\eta_n(s)}{2c\eta_n(s)}}}{\sqrt{2\pi c \eta_n(s)}} \Phi \left(-\frac{\sqrt{s-\eta_n(s)}-(s-\eta_n(s))}{\sqrt{s-\eta_n(s)}}\right)dx\\
&=    \frac{\sqrt{s-\eta_n(s)}}{C} \frac{e^{-\frac{s-\eta_n(s)}{2c\eta_n(s)}}}{\sqrt{2\pi c \eta_n(s)}} \Phi \left(-\left(1-\sqrt{s-\eta_n(s)}\right)\right).
\end{align*}
Moreover, using  the Komatsu's inequality (see \cite{ItoMc} page 17 Problem $1$), 
\begin{align*}
\Phi(-|x|) \geq \frac{2e^{-\frac{x^2}{2}}}{{\sqrt{2 \pi}}(|x|+\sqrt{x^2+4})},
\end{align*}
we get for any $n \geq T$,
\begin{align*}
 \e[| b(X_s^{(n)}) - b(X_{\eta_n(s)}^{(n)})|^q] 
&\geq \frac{C_{T,c}}{\pi C\sqrt{c}} \frac{\sqrt{s-\eta_n(s)}}{\sqrt{\eta_n(s)}}, \quad \text{for any } s \geq \frac{T}{n},
\end{align*}
where the constant $C_{T,c}$ is a constant depending only on $T$ and $c$.
Therefore, we have
\begin{align*}
 &\int_0^T \e[| b(X_s^{(n)}) - b(X_{\eta_n(s)}^{(n)})|^q] ds
 \geq C \int_{T/n}^T \frac{\sqrt{s-\eta_n(s)}}{\sqrt{\eta_n(s)}}ds 
\geq  \frac{C}{\sqrt{n}},
\end{align*}
for $n \geq \max\{ T,2 \}$.
This concludes \eqref{add}.

\end{Rem}

\begin{proof} [Proof of Lemma \ref{Lem}]  
Since $b$ is bounded, it is sufficient to prove (\ref{bxsbxeta}) for $q=1$.
Let $(b_N^{(i)})$ be an $\mathcal{A}$-approximation sequence of $b^{(i)}$ for each $i=1, \ldots, d$.
From Corollary \ref{exp_3}, we have
\begin{align}\label{dis_esti}
      &\int_0^T\e[| b^{(i)}(s, X_s^{(n)}) - b^{(i)}(s, X_{\eta_n(s)}^{(n)})|]ds \nonumber \\
&=    \int_0^{\frac{T}{n}}\e[| b^{(i)}(s,X_s^{(n)}) - b^{(i)}(s,X_{\eta_n(s)}^{(n)})|]ds + \int_{\frac{T}{n}}^T\e[| b^{(i)}(s,X_s^{(n)}) - b^{(i)}(s,X_{\eta_n(s)}^{(n)})|]ds \nonumber \\
&\leq    \frac{C}{n} + \lim_{N\rightarrow \infty}\int_{\frac{T}{n}}^T\e[| b_N^{(i)}(s, X_s^{(n)}) - b_N^{(i)}(s, X_{\eta_n(s)}^{(n)})|]ds.
\end{align}
So we estimate the second part of (\ref{dis_esti}).
Since $W_s-W_{\eta_n(s)}$ and $X_{\eta_n(s)}^{(n)}$ are independent, we have for $i=1, \cdots d$,
\begin{align}\label{equa_1}
      &\e\left[\left| b_N^{(i)}(s,X_s^{(n)}) - b_N^{(i)}(s,X_{\eta_n(s)}^{(n)}) \right| \right] \nonumber\\
&=    \e\Big[\Big| b_N^{(i)} \left(s,X_{\eta_n(s)}^{(n)} + (s-\eta_n(s)) b(\eta_n(s),X_{\eta_n(s)}^{(n)})+\sigma(\eta_n(s),X_{\eta_n(s)}^{(n)})(W_s-W_{\eta_n(s)})\right) \nonumber\\
&\qquad - b_N^{(i)}(s,	X_{\eta_n(s)}^{(n)}) \Big| \Big] \nonumber\\
&=    \int_{\real^d}dx\int_{\real^d}dy \left| b_N^{(i)}\left(s, x + (s-\eta_n(s)) b(\eta_n(s),x)+\sigma(\eta_n(s),x) y \right) - b_N^{(i)}(s,x) \right| \notag \\
&\times  p^{(n)}_{\eta_n(s)}(x) \left(\frac{1}{2\pi(s-\eta_n(s))}\right)^{d/2} \exp\left[-\frac{|y|^2}{2(s-\eta_n(s))}\right].
\end{align}
From the Gaussian bound condition for $p^{(n)}_{\eta_n(s)}$, there exists positive constant $C\geq 1$ and $ c>0$ such that the last term of (\ref{equa_1}) is less than
\begin{align}\label{equa_2}
      & C\int_{\real^d}dx\int_{\real^d}dy \left| b_N^{(i)}\big( s, x + (s-\eta_n(s))b(\eta_n(s),x)+\sigma(\eta_n(s),x) y \big) - b_N^{(i)}(s,x) \right| \nonumber\\ 
&\times \left(\frac{1}{ \eta_n(s)}\right)^{d/2} \exp\left[-c\frac{|x-x_0|^2}{2\eta_n (s)}\right] \left(\frac{1}{s-\eta_n(s)}\right)^{d/2} \exp\left[-\frac{|y|^2}{2(s-\eta_n(s))}\right].
\end{align}
Applying the change of variables $z=(s-\eta_n(s)) b(\eta_n(s),x)+\sigma(\eta_n(s),x)y$, (\ref{equa_2}) is bounded by	
\begin{align}\label{equa_3}
      & C \int_{\real^d}dx\int_{\real^d}dz\frac{\left| b_N^{(i)}\left(s, x + z \right) - b_N^{(i)}(s,x) \right|}{|\det(\sigma(\eta_n(s),x))|} \left(\frac{1}{\eta_n(s)}\right)^{d/2}\exp\left[-c\frac{ |x-x_0|^2}{2\eta_n (s)}\right]  \nonumber \\
&\times  \left(\frac{1}{s-\eta_n(s)}\right)^{d/2} \exp\left[-\frac{|\sigma^{-1}(\eta_n(s),x)(z-(s-\eta_n(s)) b(\eta_n(s),x))|^2}{2(s-\eta_n(s)) }\right]. 
\end{align}
Since $a^{-1}$ is uniformly elliptic, 
\begin{align*}
&|\sigma^{-1}(\eta_n(s),x)(z-(s-\eta_n(s))b(\eta_n(s),x))|^2\\
&=  \langle a^{-1}(\eta_n(s),x)(z-(s-\eta_n(s))b(\eta_n(s),x)) , z-(s-\eta_n(s))b(\eta_n(s),x) \rangle_{\real^d}\\
&\geq \lambda_0^{-1}|z-(s-\eta_n(s))b(\eta_n(s),x)|^2.
\end{align*}
By the inequality $|x-y|^2\geq \frac{1}{2} |x|^2 - |y|^2$ for any $x,y \in \real^d$, we have
\begin{align*}
   & -\frac{|\sigma^{-1}(\eta_n(s),x)(z-(s-\eta_n(s))b(\eta_n(s),x))|^2}{2(s-\eta_n(s)) }\\
&\leq -\frac{\lambda_0^{-1}|z|^2 }{4(s-\eta_n(s)) } +  \frac{\lambda_0^{-1}(s-\eta_n(s))^2|b(\eta_n(s),x)|^2}{2(s-\eta_n(s))} \leq -c\frac{|z|^2 }{2(s-\eta_n(s)) } + C.
\end{align*}
Using this estimate and Fubini theorem, (\ref{equa_3}) is less than
\begin{align}\label{equa_4}
      & C\int_{\real^d} dz \int_{\real^d}dx \left| b_N^{(i)}\left(s, x + z \right) - b_N^{(i)}(s,x) \right| \frac{\exp\left[-c\frac{ |x-x_0|^2}{2\eta_n (s)}\right] }{\left(\eta_n(s)\right)^{d/2}} \frac{\exp\left[-c\frac{|z|^2 }{2(s-\eta_n(s) )}\right]}{(s-\eta_n(s))^{d/2}} .
\end{align}
Since $b_N^{(i)}\left(s, x + z \right) - b_N^{(i)}(s, x) = \int_0^1 \langle z,\nabla b_N^{(i)}(s,x+\theta z)\rangle_{\real^d} d\theta $, (\ref{equa_4}) is less than
\begin{align}\label{equa_5}
     & C\int_{\real^d}dz\int_{\real^d}dx \int_0^1 d\theta \left| \langle z,\nabla b_N^{(i)}(s,x+\theta z) \rangle_{\real^d}\right| \frac{\exp\left[-c\frac{ |x-x_0|^2}{2\eta_n (s)}\right] }{\left(\eta_n(s)\right)^{d/2}} \frac{\exp\left[-c\frac{|z|^2 }{2(s-\eta_n(s) )}\right]}{(s-\eta_n(s))^{d/2}}  \notag\\
&\leq C\int_{\real^d}dz\int_{\real^d}dx \int_0^1 d\theta  |z| |\nabla b_N^{(i)}(s,x+\theta z)| \frac{\exp\left[-c\frac{ |x-x_0|^2}{2\eta_n (s)}\right] }{\left(\eta_n(s)\right)^{d/2}} \frac{\exp\left[-c\frac{|z|^2 }{2(s-\eta_n(s) )}\right]}{(s-\eta_n(s))^{d/2}} \notag\\
&\leq C\sum_{j=1}^d \int_{\real^d}dz\int_{\real^d}dy \int_0^1 d\theta  |z| |\partial_j b_N^{(i)}(s,y+x_0+\theta z)| \frac{\exp\left[-c\frac{ |y|^2}{2\eta_n (s)}\right] }{\left(\eta_n(s)\right)^{d/2}} \frac{\exp\left[-c\frac{|z|^2 }{2(s-\eta_n(s) )}\right]}{(s-\eta_n(s))^{d/2}},
\end{align}
where we use the change of variable $y= x - x_0$ in the last equation. It follows from Fubini theorem and condition $\mathcal{A}$(iii) that (\ref{equa_5}) is bounded by 
\begin{align}\label{equa_6}
      &\frac{C}{\sqrt{\eta_n(s)}} \int_{\real^d}dz\int_0^1 d\theta |z|  \frac{1}{(s-\eta_n(s))^{d/2} }\exp\left[-c\frac{|z|^2 }{2(s-\eta_n(s) )}\right] \nonumber\\
=  &\frac{C}{\sqrt{\eta_n(s)}} \int_{\real^d} \frac{|z| }{(s-\eta_n(s))^{d/2} }\exp\left[-\frac{c}{2}\frac{|z|^2 }{2(s-\eta_n(s) )}\right]\exp\left[-\frac{c}{2}\frac{|z|^2 }{2(s-\eta_n(s) )}\right]dz.
\end{align}
Since $|z| \exp\left[-\frac{c}{2}\frac{|z|^2 }{2(s-\eta_n(s) )}\right] \leq  \sqrt{\frac{2}{ec}} \sqrt{s-\eta_n(s)}$ for any $z \in \real^d$, (\ref{equa_6}) is less than
\begin{align*}
 C \sqrt{\frac{s-\eta_n(s)}{\eta_n(s)}} \leq \frac{C}{\sqrt{n\eta_n(s)}}.
\end{align*}
Therefore we have
\begin{align*} 
      \int_0^T\e[| b^{(i)}(s, X_s^{(n)}) - b^{(i)}(s, X_{\eta_n(s)}^{(n)})|]ds
\leq  \frac{C}{n} + \frac{C}{ \sqrt{n}} \int_{\frac{T}{n}}^T \frac{1}{\sqrt{\eta_n(s)}} ds
\leq \frac{C}{n} + \frac{C}{ \sqrt{n} } \int_0^T \frac{1}{\sqrt{s}}ds 
\leq \frac{C}{\sqrt{n}},
\end{align*}
which concludes the proof of Lemma \ref{Lem}.
\end{proof} 


\subsection{Proof of Theorem \ref{Main_1}}
Before proving Theorem \ref{Main_1}, we introduce some notations. Define 
\begin{align}\label{YU} 
Y_t^{(n)}=(Y_t^{(n,1)},\cdots, Y_t^{(n,d)})^{*}:=&X_t-X_t^{(n)},\:\: U_t^{(n)}:=X_t^{(n)}-X_{\eta _n(t)}^{(n)}.
\end{align}
Then by the definition of $X_t$ and $X_t^{(n)}$ we have
\begin{align*}
Y_t^{(n)}&=\int_0^t \left\{b(s,X_s) - b(\eta _n(s),X_{\eta _n(s)}^{(n)}) \right\}ds + \int_0^t \left\{\sigma(s,X_s) - \sigma(\eta _n(s),X_{\eta _n(s)}^{(n)})\right\}dW_s,\\
\langle Y^{(n,i)}&, Y^{(n,j)} \rangle _t \\
&= \sum_{k=1}^d \int_0^t \left\{\sigma_{i,k}(s,X_s) - \sigma_{i,k}(\eta _n(s),X_{\eta _n(s)}^{(n)}) \right\}\left\{\sigma_{j,k}(s,X_s) - \sigma_{j,k}(\eta _n(s),X_{\eta _n(s)}^{(n)}) \right\} ds.
\end{align*}
The following estimation is standard (see Remark 1.2 in \cite{Gyongy}). For the convenience of the reader we will give a proof below.
\begin{Lem} \label{Lem_1} \rm 
Under Assumption \ref{Ass_1} (ii) and (iii), for any $q>0$, there exist $C \equiv C(K,T,\lambda_0,d,q) $ such that
\begin{align*}
\sup_{t \in [0,T]} \e[|U_t^{(n)}|^q]\leq \frac{C}{n^{q/2}}.
\end{align*}
\end{Lem}
\begin{proof} \rm 
From the definition of $U_t^{(n)}$ and using the inequality $(\sum_{i=1}^{m} a_i)^q \leq m^{(q-1) \vee 0} \sum_{i=1}^{m} a_i^q$ for any $m\in \n, \ a_i \geq 0$ and $q > 0$, we have
\begin{align*}
      &|U_t^{(n)}|^{q} 
 =    \left(\sum_{i=1}^d \left| X_t^{(n,i)}-X_{\eta_n(t)}^{(n,i)}\right|^2 \right)^{q/2} \\
&=    \left(\sum_{i=1}^d \left|(t-\eta_n(t)) b^{(i)}(\eta _n(t),X_{\eta_n(t)}^{(n)}) + \sum_{j=1}^d\sigma_{i,j}(\eta _n(t),X_{\eta_n(t)}^{(n)})(W_t^j-W_{\eta_n(t)}^j) \right|^2 \right)^{q/2}\\
&\leq C\left(\left| t-\eta_n(t)\right|^q K^q + \sum_{j=1}^d \lambda_0^{q/2} |W_t^j-W_{\eta_n(t)}^j|^q\right), 
\end{align*}
so we have
\begin{align*}
      \e[|U_t^{(n)}|^q]
&\leq C \left(\left|t-\eta_n(t)\right|^q + \sum_{j=1}^d \e[|W_t^j-W_{\eta_n(t)}^j|^q] \right)\\
&\leq C\left(\left|t-\eta_n(t)\right|^q + \sum_{j=1}^d \left|t-\eta_n(t)\right|^{q/2}  \right)
 \leq \frac{C}{n^{q/2}}.
\end{align*}
This concludes Lemma \ref{Lem_1}.
\end{proof}

\begin{proof}[Proof of Theorem \ref{Main_1}] Inspired by the paper  \cite{Gyongy}, we will use the approximation technique of Yamada and Watanabe (see \cite{Yamada}, Theorem 1). 
For each  $\delta \in (1,\infty)$ and $\varepsilon \in (0,1)$, we can define a continuous function $\psi _{\delta, \varepsilon}: \real \to \real^+$ with $supp\: \psi _{\delta, \varepsilon}  \subset [\varepsilon/\delta, \varepsilon]$ such that
\begin{align*} 
\int_{\varepsilon/\delta}^{\varepsilon} \psi _{\delta, \varepsilon}(z) dz = 1; \quad 0 \leq \psi _{\delta, \varepsilon}(z) \leq \frac{2}{z \log \delta}, \:\:\:z > 0.
\end{align*}
We define a function $\phi_{\delta, \varepsilon} \in C^2(\real;\real)$ by
\begin{align*}
\phi_{\delta, \varepsilon}(x)&:=\int_0^{|x|}\int_0^y \psi _{\delta, \varepsilon}(z)dzdy.
\end{align*}
It is easy to verify that $\phi_{\delta, \varepsilon}$ has the following useful properties: 
\begin{align} 
&\phi'_{\delta, \varepsilon}(x) = \frac{x}{|x|}\phi'_{\delta, \varepsilon}(|x|), \text{ for any $x \in \real \setminus \{0\}$}. \label{phi1.5}\\
&0 \leq |\phi'_{\delta, \varepsilon}(x)| \leq 1, \text{ for any $x \in \real$}. \label{phi2}
\end{align}
Moreover we define function $\Phi_{\delta, \varepsilon} : \real^d \to \real$ by
\begin{align*}
\Phi_{\delta, \varepsilon}(x):=\phi_{\delta, \varepsilon}(|x|).
\end{align*}
Then we also have the following useful properties:
\begin{align}
|x| &\leq \varepsilon + \Phi_{\delta, \varepsilon}(x), \text{ for any $x \in \real^d$}. \label{phi1}\\ 
&\frac{\phi'_{\delta, \varepsilon}(|x|)}{|x|} \leq \frac{\delta}{\varepsilon}, \text{ for any $x \in \real^d \setminus\{0\}$}. \label{phi3}\\
\phi''_{\delta, \varepsilon}(\pm|x|)&=\psi_{\delta, \varepsilon}(|x|) \leq \frac{2}{|x|\log \delta}{\bf 1}_{[\varepsilon/\delta, \varepsilon]}(|x|), \text{ for any $x \in \real^d \setminus\{0\}$}. \label{phi4}
\end{align}
Note that partial differentiations of $\Phi_{\delta, \varepsilon}$ give the following: for any $x \in \real^d \setminus\{0\}$,
\begin{align}\label{part_Phi} 
\partial _{i}\Phi_{\delta, \varepsilon}(x)&=\phi'_{\delta, \varepsilon}(|x|)\frac{x_i}{|x|},\\
\partial ^2_{i}\Phi_{\delta, \varepsilon}(x)&=\phi''_{\delta, \varepsilon}(|x|)\frac{x^2_i}{|x|^2}+\phi'_{\delta, \varepsilon}(|x|) \left( \frac{|x|^2 - x^2_i}{|x|^3} \right), \nonumber\\
\partial _{i}\partial _{j}\Phi_{\delta, \varepsilon}(x)&=\phi''_{\delta, \varepsilon}(|x|)\frac{x_i x_j}{|x|^2}-\phi'_{\delta, \varepsilon}(|x|) \left( \frac{ x_ix_j}{|x|^3} \right). \nonumber
\end{align}
Notice also that all derivatives of $\phi_{\delta, \varepsilon}$ and $\Phi_{\delta, \varepsilon}$ at origin equal to $0$.
In particular, note that for any $x \in \real^d$ and $i = 1, \cdots, d$, using (\ref{phi2}) and (\ref{part_Phi}),
\begin{align*}
|\partial _{i} \Phi_{\delta, \varepsilon}(x)| \leq 1.
\end{align*}
Then It\^o's formula, (\ref{phi1.5}) and (\ref{phi1}) imply that 
\begin{align}\label{esti_Y}
      &|Y_t^{(n)}|\leq \varepsilon + \Phi_{\delta, \varepsilon}(Y_t^{(n)}) \nonumber \\
&= \varepsilon
   + \int_0^t I_s^{\delta, \varepsilon, n}ds 
   + M_t^{\delta, \varepsilon, n}
   + \frac{1}{2}\sum_{i,j=1}^d	\int_0^t \partial _{i}\partial _{j}\Phi_{\delta, \varepsilon}(Y_s^{(n)})d\langle Y^{(n,i)},Y^{(n,j)} \rangle_s ,
\end{align}
where
\begin{align*}
I_s^{\delta, \varepsilon, n} := \sum_{i=1}^d	\partial _{i}\Phi_{\delta, \varepsilon}(Y_s^{(n)})\left\{b^{(i)}(s,X_s) - b^{(i)}(\eta _n(s),X_{\eta _n(s)}^{(n)}) \right\}
\end{align*}
and
\begin{align}\label{def_M}
 M_t^{\delta, \varepsilon, n}:=\sum_{k=1}^d\sum_{i=1}^d\int_0^t \partial _{i}\Phi_{\delta, \varepsilon}(Y_s^{(n)})\left\{\sigma_{i,k}(s,X_s) - \sigma_{i,k}(\eta _n(s),X_{\eta _n(s)}^{(n)})\right\}dW^{k}_s .
\end{align}
Since $\partial _{i}\Phi_{\delta, \varepsilon}$ and $\sigma$ are bounded, $ M_t^{\delta, \varepsilon, n}$ is a martingale. Therefore the expectation of $ M_t^{\delta, \varepsilon, n}$ equals to $0$, so we only estimate the second and fourth part of (\ref{esti_Y}). First we consider the second part. 
From Assumption \ref{Ass_1b}, (\ref{phi1.5}), (\ref{phi2}) and partial differentiations of $\Phi_{\delta, \varepsilon}$, we have
\begin{align*}
      \int_0^t I_s^{\delta, \varepsilon, n}ds
&=    \sum_{i=1}^d\int_0^t \phi' _{\delta \varepsilon}(|Y_s^{(n)}|) \frac{Y_s^{(n,i)}}{|Y_s^{(n)}|}\Big\{\left(b^{(i)}(s,X_s) -  b^{(i)}(s,X_{\eta _n(s)}^{(n)}) \right) \\
&+ \left( b^{(i)}(s,X_{\eta _n(s)}^{(n)}) -  b^{(i)}(\eta _n(s),X_{\eta _n(s)}^{(n)}) \right) \Big\}ds \\
&\leq \sum_{i=1}^d\int_0^t \phi' _{\delta \varepsilon}(|Y_s^{(n)}|) \frac{Y_s^{(n,i)}}{|Y_s^{(n)}|}	\left(b^{(i)}(s,X_s) -  b^{(i)}(s,X_{\eta _n(s)}^{(n)}) \right)ds + \frac{C}{n^{\beta}} \\
&=    \int_0^t \sum_{i=1}^d \phi' _{\delta \varepsilon}(|Y_s^{(n)}|) \frac{Y_s^{(n,i)}}{|Y_s^{(n)}|}	\left(b^{(i)}(s,X_s) -  b^{(i)}(s,X_{s}^{(n)}) \right)ds \\
&+    \int_0^t \sum_{i=1}^d \phi' _{\delta \varepsilon}(|Y_s^{(n)}|) \frac{Y_s^{(n,i)}}{|Y_s^{(n)}|}	\left(b^{(i)}(s,X_s^{(n)}) -  b^{(i)}(s,X_{\eta _n(s)}^{(n)}) \right)ds
 +    \frac{C}{n^{\beta}}\\
&\leq \int_0^t \langle X_s - X_s^{(n)} , b(s,X_s) -  b(s,X_{s}^{(n)})\rangle_{\real^d} \frac{\phi' _{\delta \varepsilon}(|Y_s^{(n)}|)}{|Y_s^{(n)}|}ds\\
&+    \sum_{i=1}^d \int_0^t \left| b^{(i)}(s,X_s^{(n)}) -  b^{(i)}(s,X_{\eta _n(s)}^{(n)}) \right|ds
 +    \frac{C}{n^{\beta}}
\end{align*}
By using the one-sided Lipschitz condition (\ref{OLS_2}), we have
\begin{align}\label{esti_I}
     \int_0^t I_s^{\delta, \varepsilon, n}ds
\leq K\int_0^t | Y_s^{(n)} |ds 
 +    \sum_{i=1}^d \int_0^{T} \left| b^{(i)}(s,X_s^{(n)}) -  b^{(i)}(s,X_{\eta _n(s)}^{(n)}) \right|ds
 +    \frac{C}{n^{\beta}}.
\end{align}
Next we estimate the fourth part of (\ref{esti_Y}). Using partial differentiations of $\Phi_{\delta, \varepsilon}$, the fourth part of (\ref{esti_Y}) can be expressed by
\begin{align*}
     &\frac{1}{2}\sum_{i,j=1}^d	\int_0^t \partial _{i}\partial _{j}\Phi_{\delta, \varepsilon}(Y_s^{(n)})d\langle Y^{(n,i)},Y^{(n,j)} \rangle_s = A_t^{1,\delta, \varepsilon, n} + A_t^{2,\delta, \varepsilon, n},
\end{align*}
where 
\begin{align*}
 A_t^{1,\delta, \varepsilon, n}:=    \frac{1}{2}\sum_{i,j=1}^d \int_0^t \phi''_{\delta, \varepsilon}(|Y_s^{(n)}|)\frac{Y_s^{(n,i)} Y_s^{(n,j)}}{|Y_s^{(n)}|^2} d\langle Y^{(n,i)},Y^{(n,j)} \rangle_s
\end{align*}
and
\begin{align*}
A_t^{2,\delta, \varepsilon, n}&:=\frac{1}{2}\sum_{i=1}^d \int_0^t \phi'_{\delta, \varepsilon}(|Y_s^{(n)}|) \left( \frac{|Y_s^{(n)}|^2 - |Y_s^{(n,i)}|^2}{|Y_s^{(n)}|^3} \right) d\langle Y^{(n,i)},Y^{(n,i)} \rangle_s \\
&+ \sum_{1\leq i<j\leq d}	\int_0^t \left\{-\phi'_{\delta, \varepsilon}(|Y_s^{(n)}|)  \frac{ Y_s^{(n,i)} Y_s^{(n,j)}}{|Y_s^{(n)}|^3} \right\} d\langle Y^{(n,i)},Y^{(n,j)} \rangle_s.
\end{align*}
Here we remark that $A_t^{2,\delta, \varepsilon, n} = 0$ for $d=1$.
So we should estimate $A_t^{1,\delta, \varepsilon, n}$ and $A_t^{2,\delta, \varepsilon, n}$. By the definition of quadratic variation of $Y_t^{(n)}$, 
\begin{align*}
      A_t^{1,\delta, \varepsilon, n}
&\leq    \frac{1}{2}\sum_{k=1}^d \sum_{i,j=1}^d \int_0^t \phi''_{\delta, \varepsilon}(|Y_s^{(n)}|)\frac{|Y_s^{(n,i)}||Y_s^{(n,j)}|}{|Y_s^{(n)}|^2} \left|\sigma_{i,k}(s,X_s) - \sigma_{i,k}(\eta_n(s),X_{\eta_n(s)}^{(n)})  \right| \\
&\times \left|\sigma_{j,k}(s,X_s) - \sigma_{j,k}(\eta_n(s),X_{\eta_n(s)}^{(n)})  \right| ds,
\end{align*}
and
\begin{align*}	
       A_t^{2,\delta, \varepsilon, n}
&\leq  \frac{1}{2}\sum_{k,i=1}^d \int_0^t \phi'_{\delta, \varepsilon}(|Y_s^{(n)}|) \left( \frac{|Y_s^{(n)}|^2 - |Y_s^{(n,i)}|^2}{|Y_s^{(n)}|^3} \right) \left|\sigma_{i,k}(s,X_s) - \sigma_{i,k}(\eta_n(s),X_{\eta_n(s)}^{(n)})  \right|^2 ds \\
&\ \ + \sum_{k=1}^d\sum_{1\leq i<j\leq d}	\int_0^t \phi'_{\delta, \varepsilon}(|Y_s^{(n)}|) \frac{ |Y_s^{(n,i)}|| Y_s^{(n,j)}|}{|Y_s^{(n)}|^3} \left|\sigma_{i,k}(s,X_s) - \sigma_{i,k}(\eta_n(s),X_{\eta_n(s)}^{(n)})  \right| \\
&\times \left|\sigma_{j,k}(s,X_s) - \sigma_{j,k}(\eta_n(s),X_{\eta_n(s)}^{(n)})  \right| ds.
\end{align*}
Since $\sigma$ is $(1/2+\alpha)$-H\"older continuous in space and $\beta$-H\"older continuous in time, we have
\begin{align*}
A_t^{1,\delta, \varepsilon, n}
&\leq  C \sum_{i,j=1}^d \int_0^t \phi''_{\delta, \varepsilon}(|Y_s^{(n)}|)\frac{|Y_s^{(n,i)}||Y_s^{(n,j)}|}{|Y_s^{(n)}|^2} \left\{ \left| X_s - X_{\eta_n(s)}^{(n)} \right|^{1+2\alpha} + \frac{1}{n^{2\beta}} \right\}ds \\
&\leq  C\int_0^t \phi''_{\delta, \varepsilon}(|Y_s^{(n)}|) \left\{ \left| X_s - X_{\eta_n(s)}^{(n)} \right|^{1+2\alpha} + \frac{1}{n^{2\beta}} \right\}ds \\
&\leq  C\int_0^t \phi''_{\delta, \varepsilon}(|Y_s^{(n)}|) \left\{ \left| Y_{s}^{(n)} \right|^{1+2\alpha} + \left| U_{s}^{(n)} \right|^{1+2\alpha} + \frac{1}{n^{2\beta}} \right\}ds.
\end{align*}
Similarly, we obtain
\begin{align*}
A_t^{2,\delta, \varepsilon, n}
&\leq C\int_0^t \frac{\phi'_{\delta, \varepsilon}(|Y_s^{(n)}|)}{|Y_s^{(n)}|} \left\{ \left| Y_s^{(n)} \right|^{1+2\alpha} + \left| U_s^{(n)} \right|^{1+2\alpha} + \frac{1}{n^{2\beta}} \right\} ds. \nonumber
\end{align*}
It follows from (\ref{phi2}), (\ref{phi3}) and (\ref{phi4}) that 
\begin{align}\label{esti_A1}
A_t^{1,\delta, \varepsilon, n}&\leq  C\int_0^t \frac{{\bf 1}_{[\varepsilon/\delta, \varepsilon]}(|Y_s^{(n)}|)}{|Y_s^{(n)}|\log \delta} \left\{ \left| Y_{s}^{(n)} \right|^{1+2\alpha} + \left| U_{s}^{(n)} \right|^{1+2\alpha} + \frac{1}{n^{2\beta}} \right\}ds \nonumber\\
&\leq \frac{C \varepsilon^{2\alpha} }{\log \delta} + \frac{C \delta }{\varepsilon \log \delta} \int_0^{T} |U_s^{(n)}|^{1+2\alpha}ds + \frac{C\delta }{\varepsilon (\log \delta) n^{2\beta}}
\end{align}
and
\begin{align*}
A_t^{2,\delta, \varepsilon, n}
&\leq C\int_0^t \left| Y_s^{(n)} \right|^{2\alpha} ds  + \frac{C\delta }{\varepsilon} \int_0^{T}  \left| U_s^{(n)} \right|^{1+2\alpha} ds + \frac{C\delta}{ \varepsilon n^{2\beta}}.
\end{align*} 
Let $\tau$ be a stopping time with $\tau \leq T$. 
Define $Z_t^{(n)} := |Y_{t \wedge \tau}^{(n)}|$ and for any $\alpha \in [0,1/2]$,
\begin{align*}
     R(\alpha,\delta, \varepsilon, n)
:=  \varepsilon +\frac{C \varepsilon ^{2\alpha} }{\log \delta} + \frac{C \delta }{\varepsilon \log \delta} \int_0^{T} \left|U_s^{(n)}\right|^{1+2\alpha}ds + \frac{C\delta }{\varepsilon (\log \delta) n^{2\beta}}
\end{align*}
and
\begin{align*}
     S(\alpha,\delta, \varepsilon, n)
:=  \frac{C\delta }{\varepsilon} \int_0^{T} \left| U_s^{(n)} \right|^{1+2\alpha} ds + \frac{C\delta}{ \varepsilon n^{2\beta}}.
\end{align*}
Then we consider the following two cases.

Case $1$ ($d \geq 2$ and $\alpha =1/2$):
In this case, gathering the above estimates, we have
\begin{align}\label{d_ineq_1}
     Z_t^{(n)}
\leq C\int_0^t Z_s^{(n)} ds + \sum_{i=1}^d \int_0^{T} \left| b^{(i)}(s,X_s^{(n)}) -  b^{(i)}(s,X_{\eta _n(s)}^{(n)}) \right|ds \nonumber\\
\displaystyle + \frac{C}{n^{\beta}} + M_{t \wedge \tau}^{\delta, \varepsilon, n} + R(1/2,\delta, \varepsilon, n) +S(1/2,\delta, \varepsilon, n).
\end{align}
We choose $\delta =2$ and $\varepsilon = n^{-1/2}$. Then for any $\alpha \in (0,1/2]$, we obtain	
\begin{align*}
      R(\alpha,2,n^{-1/2}, n)
&\leq \frac{C}{\sqrt{n}} + C\sqrt{n} \int_0^T|U_s^{(n)}|^{1+2\alpha}ds+\frac{C}{n^{2\beta-1/2}}
\end{align*}
and 
\begin{align*}
      S(\alpha,2,n^{-1/2}, n)
&\leq C\sqrt{n} \int_0^T|U_s^{(n)}|^{1+2\alpha}ds+\frac{C}{n^{2\beta-1/2}}.
\end{align*}
Notice that $2\beta -1/2 \geq 1/2$. It follows from Lemma \ref{Lem_1} with $q=1+2\alpha$ that for any $\alpha \in (0,1/2]$, 	
\begin{align}\label{RS_est}
     \e[R(\alpha,2,n^{-1/2}, n)],\; \e[S(\alpha,2,n^{-1/2}, n)]
\leq \frac{C}{n^{\alpha}}.
\end{align}
Recall $\alpha = 1/2$. By using the above estimate and Lemma \ref{Lem}, we obtain 
\begin{align*}
     \e[Z_t^{(n)}]
\leq C\int_0^t \e 	[Z_s^{(n)}]ds + \frac{C}{\sqrt{n}}.
\end{align*}
By Gronwall's inequality, we have
\begin{align*}
      \e[Z_t^{(n)}]
&\leq \frac{C}{\sqrt{n}}.
\end{align*}
Therefore from dominated convergence theorem, we complete the statement taking $t \rightarrow T$.

Case $2$ ($d=1$): As remarked before that $A_t^{2,\delta, \varepsilon, n}=0$. From (\ref{esti_I}) and (\ref{esti_A1}), we have
\begin{align}\label{ineq_1}
     Z_t^{(n)}
\leq C \int_0^t Z_s^{(n)} ds + \int_0^{T} \left| b(s,X_s^{(n)}) -  b(s,X_{\eta _n(s)}^{(n)}) \right|ds + \frac{C}{n^{\beta}} + M_{t \wedge \tau}^{\delta, \varepsilon, n} + R(\alpha, \delta, \varepsilon, n).
\end{align}
For $\alpha \in (0,1/2]$, we can prove the statement in (\ref{est_l1}) in the same way as Case $1$ by taking $\delta =2$ and $\varepsilon = n^{-1/2}$. For $\alpha =0$, we choose $\delta = n^{1/3}$ and $\varepsilon = (\log n)^{-1}$. Then we have
\begin{align*} 
     R(0, n^{1/3},(\log n)^{-1}, n)
\leq \frac{C}{\log n} + Cn^{1/3} \int_0^T|U_s^{(n)}|ds +\frac{C}{n^{2\beta-1/3}}
\end{align*}
and so we get
\begin{align} \label{R_0}
     \e[R(0, n^{1/3},(\log n)^{-1}, n)]
\leq \frac{C}{\log n}.
\end{align}
From Lemma \ref{Lem}, \ref{Lem_1} and (\ref{R_0}), we have 
\begin{align*} 
      \e[Z_t^{(n)}]
&\leq C \int_0^t \e[Z_s^{(n)}]ds + \frac{C}{\log n}.
\end{align*}
Hence by Gronwall's inequality we see that
\begin{align*} 
      \e[Z_t^{(n)}]
&\leq \frac{C}{\log n}.
\end{align*}
Therefore from dominated convergence theorem, we obtain (\ref{est_l1}) for $\alpha =0$ as taking $t \rightarrow T$.
\end{proof}


\subsection{Proof of Theorem \ref{Main_2}}
Recalling (\ref{YU}), we define $ V_t^{(n)}:=\sup _{0\leq s \leq t}|Y_s^{(n)}|$. To estimate the expectation of $V_t^{(n)}$, we use (\ref{d_ineq_1}) and therefore we need to calculate the expectation of $\sup_{0\leq s \leq t}|M_s^{\delta, \varepsilon, n}|$. We use the notation $\widetilde{C}$ for a positive constant instead of $C$. This constant $\widetilde{C}$ can depend on $K, T, \alpha$ and $\beta$ while the constant $C$ can be depend on $K,T,\lambda_0,x_0,\beta$ and $d$.
For any $d \in \n$, by using (\ref{def_M}) and Burkholder-Davis-Gundy's inequality we have
\begin{align*} 
      &\e\left[\sup_{0\leq s \leq t} \left|M_s^{\delta, \varepsilon, n}\right|\right]
 \leq \widetilde{C}\e[\langle M^{\delta, \varepsilon, n} \rangle_t^{1/2} ]\\
&= \widetilde{C}\e\left[\left( \sum_{k=1}^d\int_0^t \left|\sum_{i=1}^d \partial_{i}\Phi_{\delta, \varepsilon}(Y_s^{(n)}) \left\{\sigma_{i,k}(s,X_s) - \sigma_{i,k}(\eta_n(s),X_{\eta_n(s)}^{(n)})\right\} \right|^2ds \right)^{1/2} \right].
\end{align*}
Since $\partial_{i}\Phi_{\delta, \varepsilon}$, $(i= 1, \cdots, d)$ are bounded and $\sigma$ is $1/2+\alpha$-H\"older continuous in space and $\beta$-H\"older continuous in time, we have
\begin{align}\label{sup_M} 
      &\e\left[\sup_{0\leq s \leq t} \left|M_s^{\delta, \varepsilon, n}\right|\right] \nonumber\\ 
&\leq \widetilde{C}\e\Bigg[\Big( \sum_{i,k=1}^d\int_0^t \Big\{\left|\sigma_{i,k}(s,X_s) - \sigma_{i,k}(s,X_s^{(n)})\right|^2 + \left|\sigma_{i,k}(s,X_s^{(n)}) - \sigma_{i,k}(s,X_{\eta_n(s)}^{(n)})\right|^2 \nonumber\\ 
&+ \left|\sigma_{i,k}(s,X_{\eta_n(s)}^{(n)}) - \sigma_{i,k}(\eta_n(s),X_{\eta_n(s)}^{(n)})\right|^2 \Big\}ds \Big)^{1/2} \Bigg] \nonumber\\
&= \widetilde{C}\e\Bigg[\Big(\int_0^t \Big\{\left|\sigma(s,X_s) - \sigma(s,X_s^{(n)})\right|^2 + \left|\sigma(s,X_s^{(n)}) - \sigma(s,X_{\eta_n(s)}^{(n)})\right|^2 \nonumber\\ 
&+ \left|\sigma(s,X_{\eta_n(s)}^{(n)}) - \sigma(\eta_n(s),X_{\eta_n(s)}^{(n)})\right|^2\Big\} ds \Big)^{1/2} \Bigg] \nonumber\\
&\leq \widetilde{C}\e\left[\left(\int_0^t \Big\{|X_s - X_s^{(n)}|^{1+2\alpha} + |X_s^{(n)} - X_{\eta_n(s)}^{(n)}|^{1+2\alpha } + \left| s-\eta_n(s) \right|^{2\beta}\Big\} ds \right)^{1/2} \right] \nonumber\\
&\leq \widetilde{C}\left\{ \e[A_t^{(n)}+B_t^{(n)}] + \frac{1}{n^{\beta}} \right\},
\end{align}
where by the definition of $Y^{(n)}$ and $U^{(n)}$ given in (\ref{YU}),
\begin{align*}
A_t^{(n)}:=\left( \int_0^t |Y_s^{(n)}|^{1+2\alpha}ds  \right)^{1/2},\:\:\:
B_t^{(n)}:=\left( \int_0^t |U_s^{(n)}	|^{1+2\alpha}ds  \right)^{1/2}.
\end{align*}
From Lemma \ref{Lem_1} with $q=1+2\alpha$ and using Jensen's inequality, we have 
\begin{align}\label{esti_B}
      \e[ B_t^{(n)}]
\leq  \left( \int_0^T \e\left[ |U_s^{(n)}|^{1+2\alpha} \right] ds  \right)^{1/2} 
\leq \frac{C}{n^{1/4+\alpha/2}}.
\end{align}
Next we estimate $A_t^{(n)}$ and $V_t^{(n)} = \sup_{0\leq s \leq t}|Y_s^{(n)}|$ by the following two cases.

Case $1$ ($d\geq 2$ and $\alpha = 1/2$): 
Since $|Y^{(n)}_s| \leq V_t^{(n)}$ for all $s \leq t$,  we have 
\begin{align*}
      &\e[ A_t^{(n)}]
 =    \e\left[ \left( \int_0^t |Y_s^{(n)}|^{1+2\alpha}ds  \right)^{1/2} \right]
 \leq \e\left[ (V_t^{(n)})^{1/2}\left( \int_0^t |Y_s^{(n)}|^{2\alpha}ds  \right)^{1/2} \right].
\end{align*}
Using Young's inequality $xy \leq \frac{x^2}{2\widetilde{C}} + \frac{\widetilde{C} y^2}{2}$, for any $x,y \geq 0$ and Theorem \ref{Main_1}, as $\alpha =1/2$, we get
\begin{align}\label{d_esti_A}
      \e[ A_t^{(n)}]
 \leq \frac{1}{2\widetilde{C}}\e[V_t^{(n)}] + \frac{\widetilde{C}}{2} \int_0^T \e[|Y_s^{(n)}|]ds
 \leq \frac{1}{2\widetilde{C}}\e[V_t^{(n)}] + \frac{C}{\sqrt{n}}.
\end{align}
Therefore as $\beta \geq 1/2$, we have using (\ref{sup_M}), (\ref{esti_B}) and (\ref{d_esti_A}),
\begin{align*}
     \e\left[\sup_{0\leq s \leq t} \left|M_s^{\delta, \varepsilon, n}\right|\right]
\leq \frac{1}{2} \e[V_t^{(n)}] + C \left\{ \frac{1}{\sqrt{n}} + \frac{1}{n^{\beta}} \right\}
\leq \frac{1}{2} \e[V_t^{(n)}] + \frac{C}{\sqrt{n}}.
\end{align*}
Taking supremum in (\ref{d_ineq_1}) with $\tau =T$, we obtain
\begin{align}\label{d_ineq_3} 
     V_t^{(n)}
\leq C\int_0^t V_s^{(n)} ds &+ \sum_{i=1}^d \int_0^{T} \left| b^{(i)}(s,X_s^{(n)}) -  b^{(i)}(s,X_{\eta _n(s)}^{(n)}) \right|ds  \nonumber\\
&+ \frac{C}{n^{\beta}} +\sup_{0\leq s \leq t}|M_s^{\delta, \varepsilon, n}| + R(1/2, \delta, \varepsilon, n) + S(1/2, \delta, \varepsilon, n). 
\end{align}
From (\ref{RS_est}), (\ref{d_esti_A}) and (\ref{d_ineq_3}), we have
\begin{align*}
      \e[V_t^{(n)} ]
&\leq C\int_0^t\e[V_s^{(n)}]ds + \frac{C}{\sqrt{n}}
\end{align*}
From Gronwall's inequality we have 
\begin{align*}
       \e[V_t^{(n)} ] \leq \frac{C}{\sqrt{n}}.
\end{align*}

Case $2$ ($d=1$): For $\alpha \in (0,1/2]$, by using the same method as in Case $1$, we have that (\ref{d_esti_A}) becomes
\begin{align}\label{esti_A}
      \e[ A_t^{(n)}]
 \leq \frac{1}{2\widetilde{C}}\e[V_t^{(n)}] + \frac{\widetilde{C}}{2} \int_0^T \big(\e[|Y_s^{(n)}|]\big)^{2\alpha}ds
 \leq \frac{1}{2\widetilde{C}}\e[V_t^{(n)}] + \frac{C}{n^{2\alpha^2}}.
\end{align}
Therefore from (\ref{sup_M}), using (\ref{esti_B}) and (\ref{esti_A}) we obtain
\begin{align}\label{1_sup_M}
     \e\left[\sup_{0\leq s \leq t} \left|M_s^{\delta, \varepsilon, n}\right|\right]
\leq \frac{1}{2} \e[V_t^{(n)}] + C \left\{ \frac{1}{n^{2\alpha^2}} + \frac{1}{n^{1/4+\alpha/2}} + \frac{1}{n^{\beta}} \right\}
\leq \frac{1}{2} \e[V_t^{(n)}] + \frac{C}{n^{2\alpha^2}}
\end{align}
Taking supremum in (\ref{ineq_1}) with $\tau =T$, we have
\begin{align}\label{ineq_3} 
      V_t^{(n)}
&\leq  C\int_0^tV_s^{(n)}ds + \int_0^T | b(s,X_s^{(n)}) - b(s,X_{\eta_n(s)}^{(n)})|ds + \frac{C}{n^{\beta}} + \sup_{0\leq s \leq t}|M_s^{\delta, \varepsilon, n}| + R(\alpha, \delta, \varepsilon, n).
\end{align}
Therefore by using (\ref{RS_est}), (\ref{1_sup_M}) and applying Gronwall's inequality we have
\begin{align*}
       \e[V_t^{(n)} ] \leq \frac{C}{n^{2\alpha^2}}.
\end{align*}
For $\alpha = 0$, it follows from Theorem \ref{Main_1} that we have 
\begin{align*}
      \e[ A_t^{(n)}]
\leq  \left( \int_0^T \e\left[ |Y_s^{(n)}| \right] ds  \right)^{1/2} \leq \frac{C}{\sqrt{\log n}}.
\end{align*}
Therefore from (\ref{ineq_3}) and applying Gronwall's inequality we have 
\begin{align*}
\e[V_t^{(n)} ] \leq  \frac{C}{\sqrt{\log n}}.
\end{align*}
Hence we  finish the proof of Theorem \ref{Main_2}.

\subsection{Proof of Theorem \ref{Main_3}}
To prove Theorem \ref{Main_3}, we introduce the following Gronwall type inequality.	
\begin{Lem}[\cite{Gyongy} Lemma 3.2.] \label{Lem2_1}\rm 
Let $(Z_t)_{t \geq 0}$ be a nonnegative continuous stochastic process and set $V_t:= \sup_{s\leq t}Z_s$. Assume that for some $r>0$, $q\geq 1$, $\rho \in [1,q]$ and some constants $C_0$ and $\xi \geq 0$, 
\begin{align*}
\e[V_t^r] \leq C_0 \e\left[ \left( \int_0^t V_s ds \right)^r\right] + C_0 \e\left[ \left( \int_0^t Z^{\rho}_s ds \right)^{r/q} \right] +\xi < \infty
\end{align*}
for all $t\geq 0$. Then for each $T \geq 0$ the following statements hold.\\
(i) If $\rho=q$ then there exists a constant $C_1$ depending on $C_0,T,q$ and $r$ such that
\begin{align*}
\e[V_T^r] \leq C_1 \xi.
\end{align*}
(ii) If $r\geq q$ or $q+1-\rho < r < q $ hold, then there exists constant $C_2$ depending on $C_0, T, \rho, q $ and $r$, such that
\begin{align*}
\e[V_T^r] \leq C_2 \xi + C_2 \int_0^T\e[Z_s]ds.
\end{align*}
\end{Lem}

\begin{proof}[Proof of Theorem \ref{Main_3}]
Let $p \geq 2$. First we estimate the expectation of $\sup_{0\leq s \leq t}|M_s^{\delta,\varepsilon,n}|^p$.
By using (\ref{def_M}) and Burkholder-Davis-Gundy's inequality, for any $\delta \in (1,\infty) $ and $\varepsilon \in (0,1)$,	 we have
\begin{align} 
      \e\left[\sup_{0\leq s \leq t}|M_s^{\delta,\varepsilon,n}|^p\right]
&\leq C\e[\langle M^{\delta,\varepsilon,n} \rangle_t^{p/2} ] \nonumber\\
&\leq C\e\left[\left(\int_0^t |Y_s^{(n)}|^{1+2\alpha} + |U_s^{(n)}|^{1+2\alpha } + \left| s - \eta_n(s) \right|^{2 \beta} ds \right)^{p/2} \right]\nonumber\\
&\leq C\e\left[ \left( \int_0^t |Y_s^{(n)}|^{1+2\alpha} ds \right)^{p/2} \right] + \frac{C}{n^{p/4+p\alpha /2}} + \frac{C}{n^{p \beta}} \label{sup_Mp_1} \\
&\leq C\e\left[ \left( \int_0^t |Y_s^{(n)}|^{1+2\alpha} ds \right)^{p/2} \right]+ \frac{C}{n^{p\alpha}}\label{sup_Mp_2}.  
\end{align}
Now we estimate the expectation of $(V_t^{(n)})^p$.

Case $1$ ($d \geq 2$ and $\alpha =1/2$): We choose $\delta =2$ and $\varepsilon = n^{-1/2}$. From (\ref{d_ineq_3}), by using the inequality $(\sum_{i=1}^{m} a_i)^q \leq m^{(q-1)\vee 0} \sum_{i=1}^{m} a_i^q$ for any $m\in \n, \ a_i \geq 0$ and $q >0$, we have
\begin{align}\label{Vp_d}
      (V_t^{(n)})^p
 \leq C \Bigg\{\left( \int_0^t V_s^{(n)} ds \right)^p &+ \sum_{i=1}^d\int_0^T | b^{(i)}(s,X_s^{(n)}) - b^{(i)}(s,X_{\eta_n(s)}^{(n)})|^p ds \nonumber\\
&   +     \frac{1}{n^{p\beta}} + \sup_{0\leq s \leq t}|M_s^{\delta, \varepsilon, n}|^p + R^p(1/2, \delta,\varepsilon,n) + S^p(1/2, \delta,\varepsilon,n) \Bigg\}.
\end{align}
In the same way as in (\ref{RS_est}), for any $\alpha \in (0,1/2]$ we have
\begin{align}\label{p_RS}
\e[R^p(\alpha, 2, n^{-1/2}, n)],\; \e[S^p(\alpha, 2 ,n^{-1/2}, n)] \leq Cn^{-p \alpha}.
\end{align}
Using Lemma \ref{Lem} with $q=p$, (\ref{sup_Mp_2}), (\ref{Vp_d}) and (\ref{p_RS}) we have
\begin{align*}
      \e [(V_t^{(n)})^p]
&\leq C \e \left[ \left( \int_0^t V_s^{(n)} ds \right)^p \right] + C\e\left[ \left( \int_0^t |Y_s^{(n)}|^{2} ds \right)^{p/2} \right]+ \frac{C}{\sqrt{n}}.
\end{align*}
From Lemma \ref{Lem2_1} (i) with $r=p$, $\rho = q =2 $ and $\xi = Cn^{-1/2}$, we obtain 
\begin{align*}
      \e [(V_t^{(n)})^p]
&\leq \frac{C}{\sqrt{n}}.
\end{align*}

Case $2$ ($d=1$): From (\ref{ineq_3}), we have
\begin{align}\label{Vp_1}
      (V_t^{(n)})^p
 \leq C \Bigg\{ \left( \int_0^t V_s^{(n)} ds \right)^p &+ \int_0^T | b(s,X_s^{(n)}) - b(s,X_{\eta_n(s)}^{(n)})|^p ds \nonumber\\
&+     \frac{1}{n^{p\beta }} + \sup_{0\leq s \leq t}|M_s^{\delta,\varepsilon,n}|^p + R^p(\alpha, \delta,\varepsilon,n) \Bigg\}.
\end{align}
For $\alpha =1/2$, we can show the statement in the same way as Case $1$.\\
For $\alpha \in (0,1/2)$, we also take $\delta =2$ and $\varepsilon = n^{-1/2}$.
By using (\ref{sup_Mp_2}), (\ref{p_RS}) and Lemma \ref{Lem} with $q=p$ we have
\begin{align*}
      \e [(V_t^{(n)})^p]
&\leq C \e \left[ \left( \int_0^t V_s^{(n)} ds \right)^p \right] + C\e\left[ \left( \int_0^t |Y_s^{(n)}|^{1+2\alpha} ds \right)^{p/2} \right]+ \frac{C}{n^{p\alpha}} + \frac{C}{\sqrt{n}}.
\end{align*}
From Lemma \ref{Lem2_1} (ii) with $r=p$, $q=2$, $\rho = 1+2\alpha$ and $\xi = C n^{-p \alpha} + Cn^{-1/2}$, we have 
\begin{align*}
      \e [(V_t^{(n)})^p]
&\leq \frac{C}{n^{p\alpha}} + \frac{C}{\sqrt{n}} + C \int_0^T \e[|Y_s^{(n)}|] ds \leq C\left\{ \frac{1}{n^{p\alpha}} + \frac{1}{\sqrt{n}} + \frac{1}{n^{\alpha}} \right\} \leq \frac{C}{n^{\alpha}}.
\end{align*}
For $\alpha =0$, we choose $\delta = n^{1/3}$ and $\varepsilon = (\log n)^{-1}$. 
In the same way as in (\ref{R_0}), we have
\begin{align*}
\e[R^p(0, n^{1/3},(\log n)^{-1},n)] \leq \frac{C}{(\log n)^p}.
\end{align*}
Using Lemma \ref{Lem} with $q=p$, (\ref{sup_Mp_1}),(\ref{p_RS}) and (\ref{Vp_1}) we obtain
\begin{align*}
      \e [(V_t^{(n)})^p]
&\leq C \e \left[ \left( \int_0^t V_s^{(n)} ds \right)^p \right] + C\e\left[ \left( \int_0^t |Y_s^{(n)}| ds \right)^{p/2} \right]+ \frac{C}{(\log n)^p}.
\end{align*}
From Lemma \ref{Lem2_1} (ii) with $r=p$, $q=2$, $\rho = 1$ and $\xi = C (\log n)^{-p} $, we have
\begin{align*}
      \e [(V_t^{(n)})^p]
&\leq \frac{C}{(\log n)^p} + C \int_0^T \e[|Y_s^{(n)}|]ds \leq \frac{C}{\log n}.
\end{align*}
Hence the proof of the theorem is complete.
\end{proof}

\subsection{Proof of Corollaries \ref{Main_4} and \ref{Main_5}}

Using the same argument as in (\cite{Lemaire}, Theorem 2.1), one can establish the Gaussian bound for the density of $\tilde{X}_t^{(n)}$ for $d=1$, and for the density of $\overline{X}_t^{(n)}$ for $d \geq 1$. Hence  we can prove Corollaries \ref{Main_4} and \ref{Main_5} by using the same method as in sections 3.4 and 3.5.

\subsection{Proof of Proposition \ref{exp_1}} \label{Appe}
(i). \  It is easy to prove that $\mathcal{A}$ is a vector space over $\real$.

\noindent
(ii). \ 
Let $g: [0,T]\times \mathbb{R}^d \to \mathbb{R}$ be a bounded measurable function and $g(t,\cdot): \mathbb{R}^d \to \mathbb{R}$ is monotone in each variable separately.  Let $\rho(x)$ be a density function of the $d$-dimensional standard normal distribution, i.e. $\rho(x):= e^{-|x|^2/2}/(2\pi)^{d/2}$ and a sequence $(\rho_N)_{N \in \n}$ be defined by $\rho_N(x):=N^d\rho(Nx)$.
Finally, we set $g_N(t,x):=\int_{\real^d}g(t,y)\rho_N(x-y)dy$ and $\|g\|_{\infty} := \sup_{t \in [0,T], x \in \mathbb{R}^d}|g(t,x)|$. We will show that $(g_N)$ is a $\mathcal{A}$-approximation sequence of $g$. Indeed,  since $\int_{\mathbb{R}^d}\rho_N(y)dy = 1$, we have $|g_N(t,x)| \leq \|g\|_\infty$. Thus $(g_N)$ satisfies $\mathcal{A}$(ii). Moreover,
for any $L>0$, we have
\begin{align*}
      \int_{|x|\leq L}|g_N(t,x)-g(t,x)|dx
&\leq \int_{|x|\leq L}dx\int_{\real^d}dy|g(t,y)-g(t,x)|\rho_N(x-y)\\
&=    \int_{|x|\leq L}dx\int_{\real^d}dz|g(t,x-z)-g(t,x)|\rho_N(z)\\
&=    \int_{\real^d} dz\int_{|x|\leq L}dx|g(t,x-\frac zN)-g(t,x)|\rho(z).
\end{align*}
For each $z \in \real^d$, we write $z = (z_1, \cdots, z_d)^{*}$, $z^{(0)}=0$ and $z^{(k)} = (z_1, \cdots, z_k, 0, \cdots, 0)^{*}$ for $k=1, \cdots, d$. We have 
\begin{align*}
&\int_{|x|\leq L}|g(t,x-\frac zN)-g(t,x)|dx \\
&\leq \sum_{k=1}^d \int_{|x_1|\leq L}dx_1\cdots \int_{|x_{k-1}|\leq L}dx_{k-1} \int_{|x_{k+1}|\leq L}dx_{k+1} \cdots \int_{|x_d|\leq L} dx_d \\
& \times \int_{|x_k|\leq L} \left|g(t,x-\frac{z^{(k)}}{N}) - g(t, x-\frac{z^{(k-1)}}{N})\right|dx_k.
\end{align*}
Since $g(t,\cdot)$ is monotone in each variable, 
\begin{align*}
   \int_{|x_k|\leq L} \left|g(t,x-\frac{z^{(k)}}{N}) - g(t, x-\frac{z^{(k-1)}}{N})\right|dx_k
&= \left|\int_{|x_k|\leq L} \Big(g(t,x-\frac{z^{(k)}}{N}) - g(t, x-\frac{z^{(k-1)}}{N})\Big)dx_k\right|.
\end{align*}
By the change of variable, we have
\begin{align*}
\left|\int_{|x_k|\leq L} \Big(g(t,x-\frac{z^{(k)}}{N}) - g(t, x-\frac{z^{(k-1)}}{N})\Big)dx_k\right|
&=  \left| \left( \int_{-L-z_k/N}^{-L} + \int_{L-z_k/N}^{L} \right) g(t, x-\frac{z^{(k-1)}}{N}) dx_k \right|\\
& \leq \frac{2|z_k|\|g\|_{\infty}}{N}. 
\end{align*}
Therefore
\begin{align*}
&\sup_{t \in [0,T]}\int_{|x|\leq L}|g(t,x-\frac zN)-g(t,x)|dx
 \leq  \sum_{k=1}^d \frac{2|z_k|\|g\|_{\infty}L^{d-1}}{N}. 
\end{align*}
This implies that
\begin{align*}
      \sup_{t \in [0,T]} \int_{|x|\leq L}|g_N(t,x)-g(t,x)|dx 
\leq \int_{\real^d} \sum_{k=1}^d \frac{2|z_k|\|g\|_{\infty}L^{d-1}}{N} \rho(z)dz \to 0. 
\end{align*}
as $N\to \infty$. Thus $(g_N)_{N \in \n}$ satisfies $\mathcal{A}$(i).


Since $g(t,\cdot)$ is a monotone function in each variable separately, so is $g_N(t,\cdot)$. Using the integration by parts formula, we have
\begin{align*}
      &\int_{\real^d} |\partial_i g_N(t,x+a)| \frac{e^{-|x|^2/u}}{u^{(d-1)/2}}dx
 =    \Big| \int_{\real^{d-1}}dx_1 \cdots dx_{i-1} dx_{i+1} \cdots dx_d \int_{\real} dx_i \partial_i g_N(t,x+a) \frac{e^{-|x|^2/u}}{u^{(d-1)/2}} \Big|\\
& \leq \int_{\real^d} |g_N(t,x+a)| \frac{2|x_i|}{u} \frac{e^{-|x|^2/u}}{u^{(d-1)/2}}dx
\leq \int_{\real^d}  \frac{2\|g\|_\infty |x|}{\sqrt{u}} \frac{e^{-|x|^2/u}}{u^{d/2}}dx = 2\|g\|_\infty \int_{\mathbb{R}^d}|y|e^{-|y|^2} dy,
\end{align*}
where we use the change of variable $y = x/\sqrt{u}$ in the last equation.
This concludes $(g_N)_{N \in \n}$ satisfies $\mathcal{A}$(iii).

\noindent
(iii). \ 
Let $(g_N)$ be defined as in (ii). For each $L>0$, since $g$ is Lipschitz continuous, we have
\begin{align*}
      \int_{|x|\leq L}|g_N(t,x)-g(t,x)|dx
&\leq \int_{|x|\leq L}dx\int_{\real^d}dy|g(t,y)-g(t,x)|\rho_N(x-y)\\
&\leq C\int_{|x|\leq L}dx\int_{\real^d}dy|y-x|\rho_N(x-y)\\
&=    C\int_{|x|\leq L}dx\int_{\real^d}dz \frac{|z|e^{-\frac{z^2}{2}}}{N}
\leq  \frac{C}{N} \to 0,
\end{align*}
as $N \to \infty$.
This implies $(g_N)$ satisfying $\mathcal{A}$(i). It is straightforward to verify $(g_N)$ satisfying $\mathcal{A}$(ii). To check $\mathcal{A}$(iii), we note that from the fact $\partial_i \rho_N(x)=-N^{d+2} x_i \rho(Nx)$ and Lipschitz property of $g$, 
\begin{align*}
     |\partial_i g_N(t,x)|
&=  \Big| \int_{\real^d} g(t,y) \partial_i \rho_N(x-y)dy\Big|
=  \Big| \int_{\real^d} \{ g(t,y) - g(t,x) \}  \partial_i \rho_N(x-y)dy\Big|\\
&\leq  \int_{\real^d} N^{d+2} | y - x ||y_i-x_i| \frac{e^{-\frac{N^2|y-x|^2}{2}}}{(2\pi)^{d/2}} dy.
\end{align*}
The change of variable $x=y+ z/N$ implies that 
\begin{align*}
\int_{\real^d} N^{d+2} | y - x ||y_i-x_i| \frac{e^{-\frac{N^2|y-x|^2}{2}}}{(2\pi)^{d/2}} dy \leq \int_{\real^d} |z| \frac{e^{-|z|^2/2}}{(2\pi)^{d/2}}dz = C<\infty.
\end{align*}
Hence for any $a \in \real^d$ and $u>0$, 
\begin{align*}
    \sup_{N \in \n} \sum_{i=1}^d \int_{\real^d}\left| \partial_i g_N(t,x+a) \right| \frac{e^{-\frac{|x|^2}{u}}}{u^{(d-1)/2}}dx
\leq C \int_{\real^d} \frac{e^{-\frac{|x|^2}{u}}}{u^{(d-1)/2}}dx
\leq C \sqrt{u}
\end{align*}
holds with constant $C$ which is independent of $a$ and $u$.
This concludes $(g_N)_{N \in \n}$ satisfying $\mathcal{A}$(iii).

\textbf{Acknowledgment.}
The authors are very grateful to Professor Arturo Kohatsu-Higa for suggesting this problem, and fruitful discussions and thank to Aur\'elien Alfonsi and  our Laboratory members for good advice. We also thank the referees for their comments which improve the readability of the paper.




\end{document}